\documentclass[12pt, reqno]{amsart}

\usepackage{amsmath, amsthm, amssymb}
\usepackage{url}
\usepackage[shortlabels]{enumitem}
\usepackage{comment}

\usepackage{amsfonts}
\usepackage{bm}

\usepackage{tikz}
\usetikzlibrary{arrows, cd}
\usepackage[all,cmtip]{xy}
\usepackage{mathtools}
\usepackage{adjustbox}
\usepackage{fullpage}
\usepackage{diagbox}
\usepackage[symbol]{footmisc}

\usepackage{hyperref}

\newcommand{\AAA}{\mathbf A}

\newcommand{\nn}{\mathbf n}
\newcommand{\rr}{\mathbf r}
\newcommand{\dd}{\mathbf d}

\def\cC{{\mathcal C}}

\def\cE{{\mathcal E}}

\def\cI{{\mathcal I}}

\def\PC{\operatorname{PC}}
\def\rank{\operatorname{rank}}

\newcommand{\IP}{\operatorname{IP}} 

\newcommand{\der}{\operatorname{der}}

\newcommand{\VO}{\operatorname{VO}}

\newtheorem{theorem}{Theorem}

\newtheorem{lemma}[theorem]{Lemma} 
\newtheorem{corollary}[theorem]{Corollary}

\theoremstyle{definition}

\newtheorem{example}[theorem]{Example}

\theoremstyle{remark}

\numberwithin{equation}{section}

\newcommand{\SG}{\operatorname{SG}}
\newcommand{\LG}{\operatorname{LG}}

\begin{document}

\title{The small growth invariants \\ of Goursat distributions}

\author[S.\ J.\ Colley]{Susan Jane Colley}
\address{Department of Mathematics,
Oberlin College, Oberlin, Ohio 44074, USA}
\email{scolley@oberlin.edu}

\author[G.\ Kennedy]{Gary Kennedy}
\address{Ohio State University at Mansfield, 1760 University Drive,
Mansfield, Ohio 44906, USA}
\email{kennedy@math.ohio-state.edu}

\author[C.\ Shanbrom]{Corey Shanbrom}
\address{California State University, Sacramento, 6000 J St., Sacramento, CA 95819, USA}
\email{corey.shanbrom@csus.edu}


\begin{abstract}
This is the second of a pair of papers devoted to the local invariants of Goursat distributions.
The study of these distributions naturally
leads to a tower of spaces over an arbitrary surface,
called the monster tower,
and thence to connections with the topic
of singularities of curves on surfaces.
In the prior paper we studied
 those invariants of Goursat distributions
akin to those of curves on surfaces,
which we call structural invariants.
In this paper we study invariants arising from
the small growth sequence of a Goursat distribution,
and relate them to the the structural invariants.
\end{abstract}



\date{\today}
\maketitle


\section{Introduction}

A \emph{distribution} $D$ on a manifold $M$ is a subbundle of the tangent bundle $TM$.
It is called \emph{Goursat}
if the \emph{Lie square sequence}
\[ 
D=D_{1} \subseteq D_{2} \subseteq D_{3} \subseteq \cdots 
\]
(as defined in Section~\ref{smallgrowthsection})
is a sequence of vector bundles for which
$$
\rank D_{i+1} =1+\rank D_i
$$
until one reaches the full tangent bundle.
This is the second of a pair of papers devoted to the local invariants of Goursat distributions.
In the first paper \cite{MR4887124} we gave an account of those invariants that are akin to well-known
invariants in the theory of curves on surfaces.
Here we turn our attention to invariants stemming more directly from the definition of
Goursat distributions, which have no existing counterpart in the theory of curves on surfaces;
we will call them \emph{small growth invariants}.
We continue, however, to use the connection with that other theory
in order to analyze them and work out recursive schemes for computing them.
The invariants of the earlier paper, when applied in the context of curves on surfaces,
are all what singularity theorists call \emph{complete topological invariants}.
The small growth invariants are somewhat coarser:
this is because curves that are not of the same topological type may define the
same Goursat distribution.
Nevertheless, they seem to be more challenging to calculate.

Figure~\ref{invdiagram}
shows most of the invariants that are analyzed in the two papers.
(In the body of this paper we will introduce several other invariants,
including the table of values for the quantities $e_{hi}$ of Section~\ref{secondstructure},
its associated \emph{$b$ vector}, and the \emph{proximity diagram}.)
Figure~\ref{exdiagram} presents an example, using the
same layout.

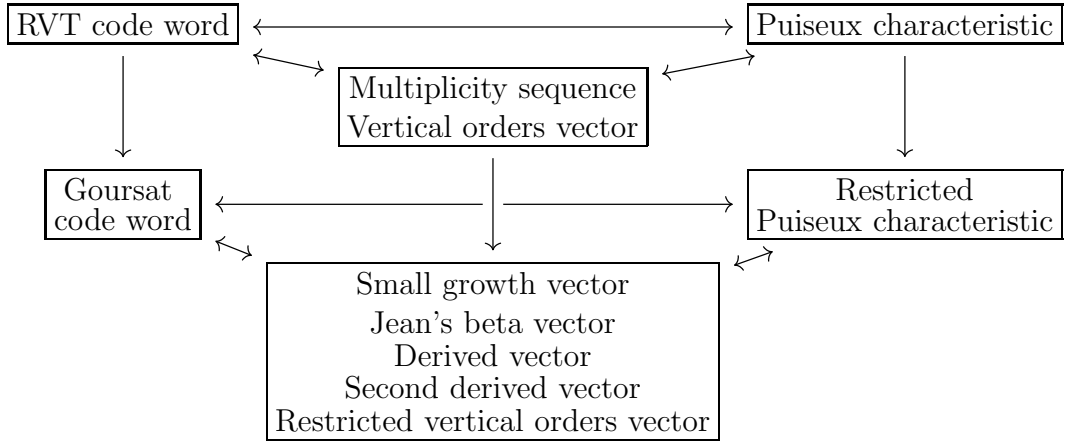
\begin{figure} 
\[   \begin{tikzcd}[sep=0pt]
\fbox{RVT code word} \arrow[rr, leftrightarrow]
\arrow[dd] \arrow[dr, leftrightarrow] &   &
\fbox{Puiseux characteristic} 
\arrow[dl, leftrightarrow] \arrow[dd] \\  & 
\fbox{$
\substack{\text{\normalsize{Multiplicity sequence}} \\[3pt] \text{\normalsize{Vertical orders vector}}}
$}
  &    \\
\fbox{$
\substack{\text{\normalsize{Goursat} }\\[2pt] \text{\normalsize{code word}}} 
$}
\arrow[rr, leftrightarrow]  \arrow[dr, leftrightarrow] &   & 
\fbox{$
\substack{\text{\normalsize{Restricted}} \\[2pt] \text{\normalsize{Puiseux characteristic}}} 
$}
\arrow[dl, leftrightarrow] \\   &   
\fbox{$
\substack{\text{\normalsize{Small growth vector}} \\[3pt] \text{\normalsize{Jean's beta vector}} \\[3pt] \text{\normalsize{Derived vector}} \\[3pt] \text{\normalsize{Second derived vector}} \\[3pt] \text{\normalsize{Restricted vertical orders vector}}} 
$}
\arrow[uu, leftarrow, crossing over]  &   
\end{tikzcd} \]
\caption{The invariants listed in the top five boxes
were studied in our prior paper \cite{MR4887124};
those listed in the bottom box are invariants of the small growth sequence
of a Goursat germ. See Figure~\ref{exdiagram} for an example.}
\label{invdiagram}
\end{figure}

\begin{figure}
\begin{center}
\includegraphics[scale=0.85]{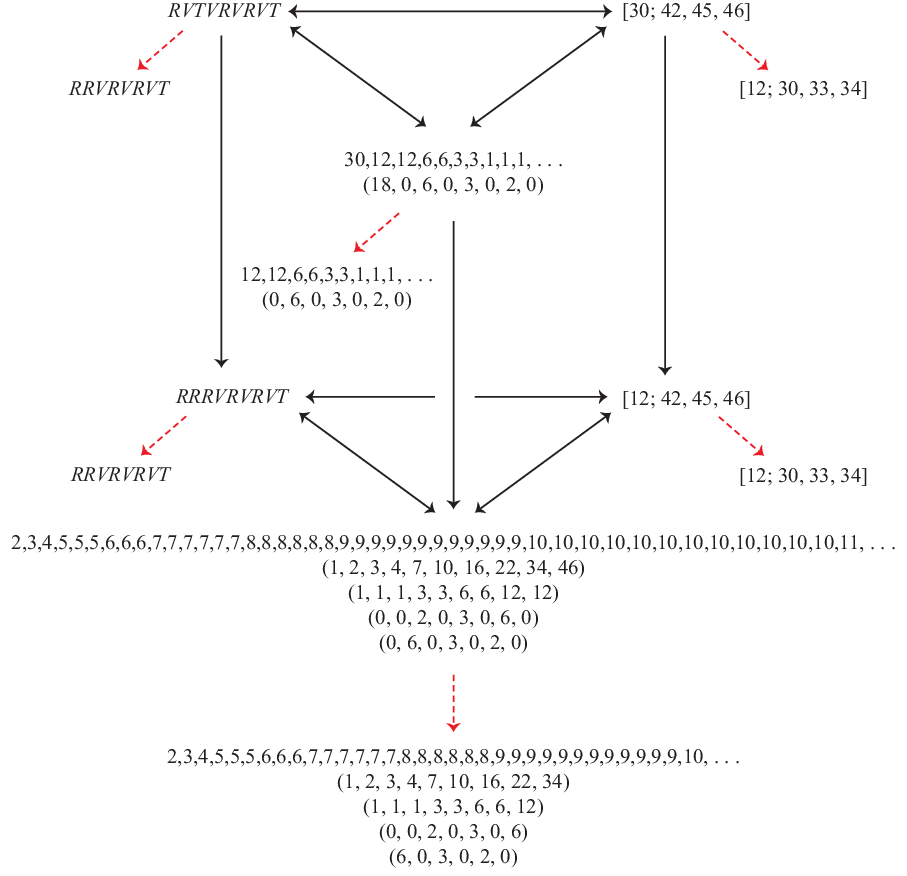}
\end{center}
\caption{An example of the invariants of Figure~\ref{invdiagram}.
The dashed arrows indicate compatible front-end recursions.}
\label{exdiagram}
\end{figure}

As explained in Section~\ref{3settings},
we give our explanations and arguments simultaneously in three settings:
smooth manifolds, complex manifolds, and nonsingular algebraic varieties.
We begin in Section~\ref{smallgrowthsection} by reviewing the definitions
of the Lie square sequence, Goursat distributions, the small growth sequence
and small growth vector, Jean's beta vector and its derived vectors,
and the degree of nonholonomy.
In Section~\ref{recollect} we briefly recall the content of our earlier paper 
\cite{MR4887124}. We merely name the most essential vocabulary from that paper;
to fully understand what we are doing here, one will need to read many of its details.
Since we make extensive use of the coordinate systems on the standard
charts of the monster spaces,
in Section~\ref{ccvf} we recall their construction; we also introduce two 
sequences of standard vector fields.
In the rather technical Section~\ref{ldlb}, we work out formulas for the Lie
brackets of these vector fields,
and find a particular basis in which these formulas look relatively simple.
Section~\ref{fovo} explicates the notions of focal order and vertical order
for functions on the monster spaces.
Section~\ref{secondstructure} presents Theorem~\ref{structuretheorem2},
our main technical result; it provides detailed information about
the sheaves in the small growth sequence, leading to inequalities
(in Corollary~\ref{comparison})
that compare structural invariants and small growth invariants.
We want to show that in fact we have equalities,
and to do so we need to exhibit specific sections of the sheaves.
We do exactly that in Section~\ref{pathway};
the basic idea is that we want to avoid certain cancellations of leading terms.
The remaining sections are the payoff.
Section~\ref{tsgi} briefly explains the simple relationships
between three structural invariants and their three small growth counterparts.
Section~\ref{recursions} discusses how one can calculate these six invariants
as functions of the RVT or Goursat code words
via recursion.
Our line of argument naturally leads to front-end recursions,
but we also remark on back-end recursions, most notably that of Jean
for his beta vector \cite{MR1411581}.
We conclude in Section~\ref{NHPC} by showing that the degree of nonholonomy for
a Goursat distribution coincides with the last entry in its associated Puiseux characteristic;
it seems that this was never noticed previously.

We thank Richard Montgomery, Piotr Mormul, Lee McEwan, and Justin Lake
for valued feedback and enormous patience.


\newpage
\section{The three settings}
\label{3settings}

We begin with a space $M$ of dimension $m \ge 2$, by which we mean either
\begin{enumerate}
\item \label{settingrm}
a smooth manifold, or
\item \label{settingcm}
a complex manifold, or
\item \label{settingav}
a nonsingular algebraic variety
over an algebraically closed field
of characteristic 0.
\end{enumerate}
All the constructions of this paper are understood
to be in the chosen setting, e.g.,
in setting~\ref{settingcm}
we work with holomorphic sections of holomorphic bundles,
with $m$ denoting the complex dimension of $M$,
whereas in setting~\ref{settingav}
we work with algebraic bundles and sections.
In a similar way, we use the word ``surface'' to mean
either a smooth manifold of dimension 2, a complex
manifold of dimension 2, or a nonsingular algebraic
surface over an algebraically closed field
of characteristic 0,
with a similar convention
for the word ``curve.''
We remark that although in our prior paper \cite{MR4887124}
we assumed the first setting,
all of the results obtained there are equally valid in all three settings.


\section{Small growth} \label{smallgrowthsection}

We repeat here selected definitions from \cite{MR4887124}.
Let $D$ be a distribution on $M$, i.e., 
a subbundle of its tangent bundle $TM$.
Let $\cE$ be its sheaf of sections,
which is a subsheaf of the sheaf $\Theta_M$
of sections of $TM$.
In other words, let $\Theta_M$
be the sheaf of vector fields on $M$,
and let $\cE$ be the subsheaf of
vector fields
tangent to $D$.

The \emph{Lie square} of $\cE$ is
\[
\cE_{2}=[\cE,\cE],
\]
meaning the subsheaf of $\Theta_M$ whose sections are generated by
Lie brackets of sections of $\cE$ and the sections of $\cE$ itself.
Note that in general the rank of $\cE_{2}$
may vary from point to point.
If, however, $\cE_{2}$
is the sheaf of sections of a distribution $D_{2}$,
then this rank is constant.
Beginning with $\cE_{1}=\cE$,
recursively we define $\cE_{i+1}$ to be the Lie square of $\cE_{i}$,
and we call 
\begin{equation} \label{Liesquaresheaves}
\cE_{1} \subseteq \cE_{2} \subseteq \cE_{3} \subseteq \cdots 
\end{equation}
the \emph{Lie square sequence}.

Let $d$ be the rank of $D$.
We say that $D$ is a \emph{Goursat distribution}
if each sheaf $\cE_{i}$ is the sheaf of sections
of a distribution $D_{i}$
and if
$$
\rank D_{i+1} = 1 + \rank D_i
$$
for $i=1,\dots,m-d$.
In particular $D_{m-d+1}$ is the tangent bundle $TM$.
For a Goursat distribution, the sequence
\[ 
D=D_{1} \subset D_{2} \subset D_{3} \subset \cdots \subset D_{m-d} \subset D_{m-d+1} = TM
\]
is also called the Lie square sequence.
Since a distribution of rank 1 is always integrable,
a Goursat distribution necessarily has rank at least 2.
As explained at the end of Section~6 of \cite{MR4887124}, the essential case for understanding all Goursat distributions is the case of rank 2.
For the remainder of the paper we assume that $D$ is a Goursat distribution with $d=2$.

Again denoting the sheaf of sections of $D$ by $\cE$, we consider the
\emph{small growth sequence}:
$$
\cE = \cE^1 \subseteq \cE^2 \subseteq \cE^3 \subseteq \cdots
$$
defined by
$$
\cE^j = [\cE, \cE^{j-1}],
$$
meaning the subsheaf of $\Theta_M$ whose sections are generated by
Lie brackets of sections of $\cE$ with sections of $\cE^{j-1}$
and by the sections of $\cE^{j-1}$.
Note how this differs from the Lie square sequence in (\ref{Liesquaresheaves}): at each step we form
Lie brackets with vector fields from the beginning distribution.
For each point $p \in M$ we let
$\SG_i$ denote the rank of $\cE^i$ at $p$;
we call
\[
\SG(D,p) = \SG_1,\SG_2,\ldots
\]
the \emph{small growth vector}.
At each step the rank
grows at most by one, and eventually the entries of $\SG$
stabilize at the value $m$.
Simple examples show that this vector may differ from point to point of $M$.  

Jean's \emph{beta  vector} \cite{MR1411581}
\[
\beta(D,p) = (\beta_2,\beta_3, \dots,\beta_m)
\]
records the positions
in the small growth vector where the rank increases:
\[
\beta_i = \min\{j: \SG_j=i\}.
\]
The last entry $\beta_m$, which records how many steps
it takes to reach the full tangent bundle, is called
the \emph{degree of nonholonomy}.

From the definitions we immediately see that $\cE^1 = \cE_1$, $\cE^2 = \cE_2$, and 
that $\cE^3$ is a subsheaf of $\cE_3$. In fact $\cE^3 = \cE_3$. To establish the opposite containment, we consider a generator of $\cE_3$ and write it, using the Jacobi identity, as a sum of two sections of $\cE^3$:
\[
[[a, b],[c, d]] = -[c,[d,[a,b]]] + [d,[c,[a,b]]].
\]
Thus the small growth vector of a rank 2 Goursat distribution always begins with $2, 3, 4$ and  the initial components of the beta vector are $1, 2, 3$.

The differences of successive terms
\[
\der_i = \beta_i - \beta_{i-1}
\]
form the \emph{derived vector}
\[\der(D,p)=(\der_3,\der_4,\ldots,\der_m),\]
and the second differences
\[
{\der^2}_i = \der_i - \der_{i-1}
\] form the \emph{second derived vector}
\[\der^2(D,p)=({\der^2}_4,{\der^2}_5,\ldots,{\der^2}_m).\]
We note that the  derived vector must begin with a pair of $1$'s and the second derived vector begins with $0$.
\begin{example}
If the small growth vector of $D$ at $p$ is
\[
\SG(D,p) = 2,3,4,5,6,6,6,7,7,7,8,8,8,8,8,8,9,9,9,9,9,9,9,9,9,
10,10,\ldots 
\]
then
\begin{align*}
\beta(D,p) &= (1,2,3,4,5,8,11,17,26) \\
\der(D,p) &= (1,1,1,1,3,3,6,9) \\
\der^2(D,p) &= (0,0,0,2,0,3,3). 
\end{align*}
\end{example}


\section{Recollection of the first paper} \label{recollect}

This paper relies heavily on our previous paper \cite{MR4887124}.
In the following breezy recollection, we italicize vocabulary whose
definitions can be found there.
In that paper we recalled, following \cite{MR1841129} and \cite{MR2172057}, how one
assigns a \emph{Goursat code word} to the germ of a Goursat distribution
at a point. Following \cite{MR1841129},
we explained the notion of \emph{prolongation} of Goursat distributions,
which leads to a construction of the \emph{monster tower}, a tower of 
\emph{monster spaces} $S(k)$
over a surface $S$;
each $S(k)$ is a space of dimension $k+2$.
These spaces are universal for Goursat distributions: given a germ
of Goursat distribution of rank $2$ and corank $k$,
one can find a point on $S(k)$
for which the germ of the 
\emph{focal distribution} $\Delta(k)$ at that point is \emph{equivalent} to the specified germ.
The spaces $S(k)$ are naturally stratified by their
\emph{divisors at infinity} and the prolongation of these divisors;
using them we introduce \emph{RVT code words}, 
which agree with the Goursat code words if one avoids the
divisor at infinity $I_2$.
Working with \emph{focal curve germs}, we defined various 
\emph{structural invariants}\,---\,those shown in the top five boxes 
of Figure~\ref{invdiagram}\,---\,and explicated recursive schemes for calculating them.
We introduced the notion of \emph{Goursat invariants};
the present paper is devoted to further study of these invariants.

\section{Charts, coordinates, focal vector fields} \label{ccvf}

We will use the standard charts on the monster space $S(k)$ explained in Section~8 of \cite{MR4887124}.
We repeat here, mostly verbatim, our descriptions of those charts.

We begin with a specified chart $U$ on $S$ with coordinates
$\rr_0$ and $\nn_0$.
On $S(k)$ there are $2^k$ standard charts over $U$, each of which is 
a copy of $U \times \AAA^k$,
and on each such chart there are $k+2$ coordinate functions, namely
$\rr_0$ and $\nn_0$ together with
$k$ affine coordinates.
At each level $j$,
by a recursive procedure, 
two of these coordinates are
designated as \emph{active coordinates}.
One is the \emph{new coordinate}
$\nn_j$
and the other is the \emph{retained coordinate}
$\rr_j$.
In addition, for $j>0$,
a third coordinate is designated as the \emph{deactivated coordinate}
$\dd_j$.

To describe the recursive procedure,
we begin with a standard chart on $U(j)$ with coordinates
 $\nn_{j}$, $\rr_{j}$, and $\dd_{j}$
together with $j-1$ unnamed coordinates.
At each point of the chart,
the fiber of $\Delta(j)$
(except for the zero vector)
consists of tangent vectors for which
either the restriction of the differential $d\nn_{j}$ or 
that of
$d\rr_{j}$ is nonzero.
Create a standard chart at the next level by choosing one of the following two options:
\begin{itemize}
\item
Assuming the restriction of $d\rr_{j}$ is nonzero, let
$\nn_{j+1}=d\nn_{j}/d\rr_{j}$; then set
$\rr_{j+1}=\rr_{j}$ and $\dd_{j+1}=\nn_{j}$.
We call this the \emph{ordinary choice}.
\item
Assuming the restriction of $d\nn_{j}$ is nonzero, let
$\nn_{j+1}=d\rr_{j}/d\nn_{j}$; then set
$\rr_{j+1}=\nn_{j}$ and $\dd_{j+1}=\rr_{j}$.
We call this the \emph{inverted choice}.
\end{itemize}
In every standard chart the names of the coordinates are $\rr_0$, $\nn_0$, $\nn_1$,
\dots, $\nn_j$, but their meaning depends on the chart.
The charts are given names such as $\cC(oiiooi)$, where each symbol $o$
or $i$ records which choice has been made, either ordinary or inverted.

In an alternative procedure for naming coordinates,
we begin with coordinates $x^{(0)}=x$ and $y^{(0)}=y$.
At each level, the two active coordinates will be $x^{(i)}$ and
$y^{(j)}$, for some nonnegative integers $i$ and $j$.
If we create our chart at the next level by
designating $x^{(i)}$ as the retained coordinate,
then the new active coordinate is $y^{(j+1)}=dy^{(j)}/dx^{(i)}$;
if we designate $y^{(j)}$ as the retained coordinate,
then the new active coordinate is $x^{(i+1)}=dx^{(i)}/dy^{(j)}$.

As remarked in the earlier paper,
the focal sheaf $\Delta(k)$ consists of those vector fields
annihilated by the differential form
\begin{equation}\label{adjoined}
d\dd_i-\nn_i \, d\rr_i
\end{equation}
for $i=1,\dots,k$;
we call them \emph{focal vector fields}.
Since $k$ is fixed, we abbreviate to $\Delta$.
Similarly, the other sheaves $\Delta_i$ in the Lie square sequence,
which (by equation (5.3) of \cite{MR4887124}) are extensions
of focal sheaves from spaces lower in the tower,
consist of vector fields annihilated by (\ref{adjoined})
for $i=1,\dots,k-i+1$.

On each standard chart we now define two sequences of 
\emph{standard vector fields}.
The sequence $v_0$, $v_1$, \dots, $v_k$ is defined by
\begin{equation} \label{defvi}
v_i = \frac{\partial}{\partial \nn_i}.
\end{equation}
The last vector field $v_k$ is vertical,
i.e., its projection to $S(k-1)$ is the zero vector field.
For $1\le i \le k-1$,
each  $v_{i}$ projects to a vertical vector field
at level $i$ (and hence to the zero vector field at level $i-1$).
The sequence $f_0$, $f_1$, \dots, $f_k$ is defined by
recursion.
To begin, set 
\begin{equation}\label{deff0}
f_0=\frac{\partial}{\partial \rr_0}\,.
\end{equation}
If at level $i$ we make the ordinary choice, then we define
\begin{equation} \label{ordinarychoice}
f_{i} = f_{i-1} + \nn_{i} \, v_{i-1}\,.
\end{equation}
If at level $i$ we make the inverted choice, then we define
\begin{equation} \label{invertedchoice}
f_{i} = \nn_{i} f_{i-1} + v_{i-1}\,.
\end{equation}
The last vector field $f_k$ is focal, and each $f_{i}$ projects to a focal vector field
at level $i$.
\begin{example} \label{focalvfexample}
In chart $\cC(ooioii)$ we have
\begin{align*}
f_0 &= f_0 \\
f_1 &= f_0 + \nn_1 v_0 \\
f_2 &= f_0 + \nn_1 v_0 + \nn_2 v_1 \\
f_3 &= \nn_3 f_0 + \nn_1 \nn_3 v_0 + \nn_2 \nn_3 v_1 + v_2 \\
f_4 &= \nn_3 f_0 + \nn_1 \nn_3 v_0 + \nn_2 \nn_3 v_1 + v_2 + \nn_4 v_3 \\
f_5 &= \nn_3 \nn_5 f_0 + \nn_1 \nn_3 \nn_5 v_0 + \nn_2 \nn_3 \nn_5 v_1 + \nn_5 v_2 + \nn_4 \nn_5 v_3 + v_4 \\
f_6 &= \nn_3 \nn_5 \nn_6 f_0 + \nn_1 \nn_3 \nn_5 \nn_6 v_0 + \nn_2 \nn_3 \nn_5 \nn_6 v_1 + \nn_5 \nn_6 v_2 + \nn_4 \nn_5 \nn_6 v_3 +  \nn_6 v_4 +v_5 \,.
\end{align*}
If we prefer the alternative coordinate names, we use $x$, $y$, and
\begin{align*}
y'&=dy/dx \\
y''&=dy'/dx \\
x'&=dx/dy'' \\
x''&=dx'/dy'' \\
y^{(3)}&=dy''/dx'' \\
x^{(3)}&=dx''/dy^{(3)}.
\end{align*}
In these coordinates we have
\begin{align*}
f_0 &= \partial/\partial x \\
v_0 &= \partial/\partial y \\
v_1 &= \partial/\partial y' \\
v_2 &= \partial/\partial y'' \\
v_3 &= \partial/\partial x' \\
v_4 &= \partial/\partial x'' \\
v_5 &= \partial/\partial y^{(3)} \\
v_6 &= \partial/\partial x^{(3)} \\
\end{align*}
and
\begin{align*}
f_1 &= f_0 + y' v_0 \\
f_2 &= f_0 + y' v_0 + y'' v_1 \\
f_3 &= x' f_0 + y' x' v_0 + y'' x' v_1 + v_2 \\
f_4 &= x' f_0 + y' x' v_0 + y'' x' v_1 + v_2 + x'' v_3 \\
f_5 &= x' y^{(3)} f_0 + y' x' y^{(3)} v_0 + y'' x' y^{(3)} v_1 + y^{(3)} v_2 + x'' y^{(3)} v_3 + v_4 \\
f_6 &= x' y^{(3)} x^{(3)} f_0 + y' x' y^{(3)} x^{(3)} v_0 + y'' x' y^{(3)} x^{(3)} v_1 + y^{(3)} x^{(3)} v_2 + x'' y^{(3)} x^{(3)} v_3 +  x^{(3)} v_4 +v_5 \,. \\
\end{align*}
\end{example}

The recursive formulas (\ref{ordinarychoice}) and (\ref{invertedchoice})
lead to an explicit formula for $f_i$ in terms of this basis.
To state it, we let $\IP$ denote the set of positions
at which we have made the inverted choice,
and then define (for $1\le i \le j\le k$)
\begin{equation} \label{defa}
a_{ij} = \displaystyle\prod_{\substack{h \in \IP \\ 
i+1\le h\le j}}\nn_h  
\end{equation}
and (for $0\le i < j \le k$)
\begin{equation}  \label{defb}
b_{ij} =
\begin{cases}
a_{i+1,\,j} \quad\qquad \text{if $i+1 \in \IP$}\\
\nn_{i+1} a_{i+1,\,j}  \quad \text{if $i+1 \notin \IP$}. 
\end{cases}
\end{equation}
By an easy induction on $j$ we have
\begin{equation} \label{fjlc}
f_{j} = a_{1j}  f_0 + \sum_{i=0}^{j-1} b_{ij} v_i.
\end{equation}

The vector fields
\begin{equation} \label{deltaibasis}
f_{k-i+1}, v_{k-i+1},v_{k-i+2}, \dots, v_{k}
\end{equation}
form a basis of sections 
of
$\Delta_i$.
One extreme case is the focal bundle
$\Delta$, whose sections are spanned by $f_k$
and $v_k$.
At the other extreme is the full tangent bundle 
$TS(k)$, whose sections are spanned by
$f_{0}, v_{0},v_{1}, \dots, v_{k}$.
For future reference, we note an alternative
basis for $\Delta_i$.

\begin{lemma} \label{altbasis}
For each $i$, there is an alternative basis for $\Delta_i$:
\begin{enumerate}[(a)]
\item
If $k-i+2 \notin \IP$, 
then $f_{k-i+2}, v_{k-i+1}, v_{k-i+2}, \dots, v_{k}$
is an alternative basis.
\item
If $k-i+2 \in \IP$, 
then $f_{k-i+2}, f_{k-i+1}, v_{k-i+2}, \dots, v_{k}$
is an alternative basis.
\end{enumerate}
\end{lemma}
\begin{proof}
In case (a),
formula~
(\ref{ordinarychoice})
tells us that we may replace the basis
element $f_{k-i+1}$ by $f_{k-i+2}$.
Similarly in case (b),
formula~
(\ref{invertedchoice})
tells us that we may replace the basis
element $v_{k-i+1}$ by $f_{k-i+2}$.
\end{proof}

\section{Lie derivatives and Lie brackets} \label{ldlb}

In this section we work out some needed facts
about Lie derivatives and Lie brackets,
using the standard coordinate functions
and standard vector fields of Section~\ref{ccvf}.
As we will see, the formulas for Lie brackets are
somewhat irregular. In Lemma~\ref{omnibus}, however,
we will identify a particular basis mixing the standard
vertical and focal vector fields (the precise
mixture being determined by the choice of chart);
for this basis there is a uniform formula,
as expressed in property~(\ref{gbrackets}) of the lemma.
This basis will be used in Theorem~\ref{structuretheorem2}
to characterize the sections of the small growth sheaves.

If $z$ is a vector field on $S(k)$ and $a$ is a smooth function,
we use the usual notation $z(a)$ for 
the Lie derivative. Looking at a specified standard chart
on $S(k)$, we now present two lemmas about specific Lie derivatives.

\begin{lemma} \label{bijremark}
For $i<j$ we have that 
\begin{equation}
 f_j(\nn_i) = b_{ij}.
\end{equation}
\end{lemma}
\begin{proof}
This is an immediate consequence of formula~(\ref{fjlc}).
\end{proof}

\begin{lemma} \label{vkfkmonomial}
Suppose
that the function $a$ is a monomial in the standard 
coordinates. Then $v_k(a)$ and $f_k(a)$
are linear combination of monomials with positive coefficients.
\end{lemma}
\begin{proof}
Formulas (\ref{defvi}) and (\ref{deff0}) tell us that 
$f_0, v_0, v_1, \dots, v_k$
form a basis of vector fields, namely
the partial derivatives 
with respect to the
coordinate system $\rr_0, \nn_0, \dots, \nn_k$.
Formula (\ref{fjlc}) tells us how to express
$f_k$ as a linear combination of this basis,
and definitions (\ref{defa}) and (\ref{defb})
guarantee that the coefficients in this linear combination are
monomials with positive coefficients.
\end{proof}

Turning to Lie brackets,
we note that
we will make frequent use of the
product rule
for a function $a$ and a pair $w$, $z$ of vector fields:
\begin{equation} \label{productrule}
[w,az]=a[w,z]+w(a)z\,.
\end{equation}
Looking at our standard vector fields, 
we first present an example.

\begin{example}  \label{standardbracketsexample}
Continuing Example~\ref{focalvfexample},
here is a table of Lie brackets in the chart
$\cC(ooioii)$, for which $\IP=\{3,5,6\}$.
The row heading indicates the vector field used in the left slot, e.g., $[v_1,f_1]=v_0$.
\begin{center}
\begin{tabular}{ c || c c c c c c c}
       & $f_0$ & $f_1$ & $f_2$ & $f_3$ & $f_4$ & $f_5$ & $f_6$ \\ \hline\hline
$v_0$ & $0$ & $0$ & $0$ & $0$ & $0$ & $0$ & $0$ \\
$v_1$ & $0$ & $v_0$ & $v_0$ & $\nn_3 v_0$ &  $\nn_3 v_0$ &  $\nn_3 \nn_5 v_0$ & $\nn_3 \nn_5 \nn_6 v_0$ \\
$v_2$ & $0$ & $0$ & $v_1$ & $\nn_3 v_1$ &  $\nn_3 v_1$ &  $\nn_3 \nn_5 v_1$ & $\nn_3 \nn_5 \nn_6 v_1$ \\
$v_3$ & $0$ & $0$ & $0$ & $f_2$ & $f_2$ & $\nn_5 f_2$ & $\nn_5 \nn_6 f_2$ \\
$v_4$ & $0$ & $0$ & $0$ & $0$ & $v_3$ & $\nn_5 v_3$ & $\nn_5 \nn_6 v_3$ \\
$v_5$ & $0$ & $0$ & $0$ & $0$ & $0$ & $f_4$ & $\nn_6 f_4$ \\
$v_6$ & $0$ & $0$ & $0$ & $0$ & $0$ & $0$ & $f_5$ \\ \hline
$f_0$ & $0$ & $0$ & $0$ & $0$ & $0$ & $0$ & $0$ \\
$f_1$ & $0$ & $0$ & $-\nn_2 v_0$ & $-\nn_2 \nn_3 v_0$ & $-\nn_2 \nn_3 v_0$ &$-\nn_2 \nn_3 \nn_5 v_0$ & $-\nn_2 \nn_3 \nn_5 \nn_6 v_0$ \\
$f_2$ & $0$ & $\nn_2 v_0$ & $0$ &  $-v_1$ & $-v_1$ & $-\nn_5 v_1$ & $-\nn_5 \nn_6 v_1$ \\
$f_3$ & $0$ & $\nn_2 \nn_3 v_0$ & $v_1$ & $0$ & $-\nn_4 f_2$ & $-\nn_4 \nn_5 f_2$ & $-\nn_4 \nn_5 \nn_6 f_2$ \\
$f_4$ & $0$ & $\nn_2 \nn_3 v_0$ & $v_1$ & $\nn_4 f_2$ & $0$ & $-v_3$ & $-\nn_6 v_3$ \\
$f_5$ & $0$ & $\nn_2 \nn_3 \nn_5 v_0$ & $\nn_5 v_1$ & $\nn_4 \nn_5 f_2$ & $v_3$ & $0$ & $-f_4$ \\
$f_6$ & $0$ &$\nn_2 \nn_3 \nn_5 \nn_6 v_0$ & $\nn_5 \nn_6 v_1$ & $\nn_4 \nn_5 \nn_6 f_2$ & $\nn_6 v_3$ & $f_4$ & $0$
\end{tabular}

\end{center}
\end{example}

We now explain how to obtain the sort of explicit formulas appearing in Example~\ref{standardbracketsexample}.
To begin,
an elementary computation shows that, for each $j$,
\begin{equation} \label{vacuous}
[v_0,f_j]=[f_0,f_j]=0.
\end{equation}
The following two lemmas give the other formulas.

\begin{lemma} \label{lblemma}
For $i>j$ we have
\begin{equation} \label{vifjvanish}
[v_i,f_j] =0.
\end{equation}
For $1 \le i\le j \le k$,
\begin{equation} \label{LBvf0}
[v_i,f_j] = 
\begin{cases}
a_{ij} f_{i-1} \quad \text{if $i\in \IP$}\\
a_{ij} v_{i-1} \quad \text{if $i\notin \IP$}.
\end{cases}
\end{equation}
\end{lemma}
\begin{proof}
In (\ref{fjlc}) we see that $f_j$ is a linear combination of
$f_0$ and $v_0, \dots, v_{j-1}$,
from which (\ref{vifjvanish}) is immediate.

We prove the formulas of (\ref{LBvf0}) by induction on $j$.
In the inductive step, we
first suppose that $j\in\IP$.
Here (\ref{invertedchoice}) and the product rule (\ref{productrule}) give us that
\begin{equation} \label{vifj1}
[v_i,f_j]
=[v_i,\nn_j f_{j-1}+v_{j-1}] 
= v_i(\nn_j)f_{j-1}+\nn_j[v_i,f_{j-1}].
\end{equation}
If $i<j$, then the first term on the right of (\ref{vifj1})
vanishes, and using
(\ref{defa})
 we find that
\begin{equation*}
[v_i,f_j]
= \nn_j[v_i,f_{j-1}]
=
\begin{cases}
\nn_j a_{i,j-1} f_{i-1} = a_{ij} f_{i-1} \quad \text{if $i\in\IP$}
\\
\nn_j a_{i,j-1} v_{i-1} = a_{ij} v_{i-1} \quad \text{if $i\notin\IP$}.
\end{cases}
\end{equation*}
In the remaining case $i=j$, the last term on the right of (\ref{vifj1})
vanishes (as a consequence of (\ref{vifjvanish})),
and thus we have 
\[
[v_j,f_j]=f_{j-1}=a_{jj}f_{j-1}.
\]

Now suppose that $j\notin\IP$.
Here (\ref{ordinarychoice}) and the product rule (\ref{productrule}) give us that
\begin{equation} \label{vifj2}
[v_i,f_j]
=[v_i,f_{j-1}+\nn_{j-1}v_{j-1}] 
= [v_i,f_{j-1}]+v_i(\nn_j)v_{j-1}.
\end{equation}
If $i<j$, then the second term on the right of (\ref{vifj2})
vanishes, so that
\begin{equation*}
[v_i,f_j]
= [v_i,f_{j-1}]
=
\begin{cases}
a_{i,j-1} f_{i-1} = a_{ij} f_{i-1} \quad \text{if $i\in\IP$}
\\
a_{i,j-1} v_{i-1} = a_{ij} v_{i-1} \quad \text{if $i\notin\IP$}.
\end{cases}
\end{equation*}
In the remaining case $i=j$, the first term on the right of (\ref{vifj2})
vanishes, and thus we have 
\[
[v_j,f_j]=v_{j-1}=a_{jj}v_{j-1}.
\]
\end{proof}
\begin{lemma}
For $1 \le i < j \le k$
\begin{equation} \label{LBvf1}
[f_i,f_j] =
\begin{cases}
-b_{ij} f_{i-1} \quad \text{if $i\in \IP$}\\
-b_{ij} v_{i-1} \quad \text{if $i\notin \IP$}.
\end{cases}
\end{equation}
For $0 \le i < j \le k$
\begin{equation} \label{LBvf2}
\begin{cases}
\nn_{i+1}[f_i, f_j] + [v_i, f_j]=0 \quad \text{if $i+1\in \IP$}\\
[f_i, f_j] + \nn_{i+1}[v_i, f_j]=0\quad \text{if $i+1\notin \IP$}.
\end{cases}
\end{equation}
\end{lemma}
\begin{proof}In view of (\ref{LBvf0}),
and using the definitions (\ref{defa}) and (\ref{defb}),
one sees that formulas (\ref{LBvf1}) and (\ref{LBvf2})
are equivalent for each $i\ge 1$.
We prove them by induction on $i$.
To provide a base case for the induction, we remark that 
when $i=0$, formula~(\ref{LBvf2}) is a consequence of (\ref{vacuous}).

Here is the inductive step of the argument:
If $i\in\IP$, then by (\ref{invertedchoice}) we have
\begin{align*}
[f_i,f_j]&=[\nn_i f_{i-1}+v_{i-1},f_j] \\
&= -f_j(\nn_i)f_{i-1}+\nn_i[f_{i-1},f_j]+[v_{i-1},f_j].
\end{align*}
The inductive hypothesis tells us that the sum of the last two terms vanishes.
Thus, by Lemma~\ref{bijremark},
\[
[f_i,f_j] = -b_{ij}f_{i-1}.
\]
If $i\notin\IP$, then by (\ref{ordinarychoice}) we have
\begin{align*}
[f_i,f_j]&=[f_{i-1}+\nn_i v_{i-1},f_j] \\
&= [f_{i-1},f_j] -f_j(\nn_i)v_{i-1}+\nn_i [v_{i-1},f_j].
\end{align*}
The inductive hypothesis tells us that the sum of the first and last terms vanishes.
Thus, by Lemma~\ref{bijremark},
\[
[f_i,f_j] = -b_{ij}v_{i-1}.
\]
\end{proof}

\begin{lemma} \label{omnibus}
On each standard chart of $S(k)$,
there is an ordered basis $g_0, g_1, \dots, g_{k+1}$
of sections of the tangent bundle
with the following properties:
\begin{enumerate}
\item \label{startbasis}
The first two elements are the standard focal and vertical vector
fields:
$g_0=f_k$ and $g_1=v_k$.
\item \label{continuebasis}
For $2 \le i \le k+1$ we have 
\begin{equation*}
g_i =
\begin{cases} 
\pm v_{k-i+1} \quad \text{if $k-i+2 \notin \IP$}\\
\pm f_{k-i+1}  \quad \text{if $k-i+2 \in \IP$},
\end{cases}
\end{equation*}
with the appropriate sign being determined by property~(\ref{gbrackets}).
The vector field $g_i$
generates  
the rank one quotient sheaf $\Delta_i/\Delta_{i-1}$.
\item \label{newdeltaibasis}
The vector fields $g_0, g_1, \dots, g_i$ form a basis of sections of 
$\Delta_i$.
\item \label{monomialeffect}
If $a$ is a monomial in the standard 
coordinates, then  
$g_0(a)$ and $g_1(a)$
are linear combinations of monomials with positive coefficients.
\item \label{gbrackets}
For $1 \le i \le k$ we have
\begin{equation*}
[g_0,g_i]= \left(\displaystyle
\prod_{\substack{h=k-i+3 \\ h \in \IP}}^k
\nn_h
\right)
g_{i+1} \, .
\end{equation*}
\end{enumerate}
\end{lemma}
%
\begin{proof}
We define $g_0$ and $g_1$ as required by property (\ref{startbasis}),
and define (for $g\ge 1$)
\begin{equation*} 
g_{i+1}=
\left(
\prod_{\substack{h=k-i+3 \\ h \in \IP}}^k
\nn_h
\right)^{\!\!\!-1}
[g_0,g_i]\,.
\end{equation*}
Although it seems that there is a denominator,
we proceed to demonstrate that we obtain
the formulas of
property (\ref{continuebasis}), as follows.

Starting with the formula in property~(\ref{gbrackets})
and making the substitutions of property~(\ref{continuebasis}),
we find that the two formulas of property~(\ref{continuebasis})
are equivalent to the following four displayed formulas;
thus these formulas provide
an inductive proof of property~(\ref{continuebasis}).
Here are the four formulas:
\begin{equation*} 
[f_k,v_{k-i+1}]=\pm \left(\displaystyle
\prod_{\substack{h=k-i+3 \\ h \in \IP}}^k
\nn_h
\right)
v_{k-i} 
\qquad \text{if $k-i+2 \notin \IP$}
\quad \text{and $k-i+1 \notin \IP$}
\end{equation*}
\begin{equation*} 
[f_k,v_{k-i+1}]=\pm \left(\displaystyle
\prod_{\substack{h=k-i+3 \\ h \in \IP}}^k
\nn_h
\right)
f_{k-i} 
\qquad \text{if $k-i+2 \notin \IP$}
\quad \text{and $k-i+1 \in \IP$}
\end{equation*}
\begin{equation*} 
[f_k,f_{k-i+1}]=\pm \left(\displaystyle
\prod_{\substack{h=k-i+3 \\ h \in \IP}}^k
\nn_h
\right)
v_{k-i} 
\qquad \text{if $k-i+2 \in \IP$}
\quad \text{and $k-i+1 \notin \IP$}
\end{equation*}
\begin{equation*} 
[f_k,f_{k-i+1}]=\pm \left(\displaystyle
\prod_{\substack{h=k-i+3 \\ h \in \IP}}^k
\nn_h
\right)
f_{k-i} 
\qquad \text{if $k-i+2 \in \IP$}
\quad \text{and $k-i+1 \in \IP$}
\end{equation*}
The first two formulas are special cases (with indeterminate signs)
of (\ref{LBvf0}); observe that to verify this we use the condition
$k-i+2 \notin \IP$
to infer that $a_{k-i+1,k}=a_{k-i+2,k}$.
The last two formulas are special cases
of (\ref{LBvf1}); to verify this, we use the condition
$k-i+2 \in \IP$, which tells us that
in (\ref{defb}) we should read the top line
to infer that 
$b_{k-i+1,k}=a_{k-i+2,k}$.

To obtain property (\ref{continuebasis}), 
we compare
the basis
\begin{equation*}
f_{k-i+2}, v_{k-i+2}, \dots, v_{k}
\end{equation*}
for $\Delta_{i-1}$ obtained from (\ref{deltaibasis})
with the alternative basis for $\Delta_{i}$
exhibited in Lemma~\ref{altbasis}.
Property (\ref{newdeltaibasis})
is now immediate, and property (\ref{monomialeffect})
is a restatement of Lemma~\ref{vkfkmonomial}.
\end{proof}

\begin{example} \label{specialbasisexample}
On chart $\cC(ooioii)$,
using the results from Example~\ref{standardbracketsexample},
we have
\begin{align*}
g_0 &= f_6 \\
g_1 &= v_6 \\
g_2&= [g_0,g_1] = -f_5 \\
g_3 &= [g_0,g_2] = -f_4 \\
g_4 &= \tfrac{1}{\nn_6} [g_0,g_3] = -v_3 \\
g_5 &= \tfrac{1}{\nn_5\nn_6} [g_0,g_4] = f_2 \\
g_6 &= \tfrac{1}{\nn_5\nn_6} [g_0,g_5] = v_1 \\
g_7 &= \tfrac{1}{\nn_3\nn_5\nn_6} [g_0,g_6] = -v_0. \\
\end{align*}
\end{example}

From these explicit calculations on each standard chart,
we can now draw a global conclusion.
Recall from Section~9 of \cite{MR4887124}
that the monster space $S(k)$
carries \emph{divisors at infinity} $I_j$
for $2\le j \le k$.
Let
$\cI_j$ denote the ideal sheaf of $I_j$.

\begin{corollary}\label{oldstructuretheorem1}
For each $i$ with $3 \le i \le k$ we have
\begin{equation} \label{structureequation1}
[\Delta, \Delta_i] = \Delta_i + \left(\displaystyle\prod_{j=k-i+3}^k \cI_j\right)\Delta_{i+1}.
\end{equation}
\end{corollary}

\begin{proof}
Consider a chart specified by a word in the symbols
$o$ and $i$.
In this chart the divisor at infinity $I_j$
is the hyperplane $\nn_j=0$ if $j\in\IP$,
but 
if $j\notin\IP$ then $I_j$
does not meet the chart.
Therefore in this chart the ideal appearing in (\ref{structureequation1})
is
\begin{equation} \label{principal}
\prod_{\substack{j \in \IP \\ k-i+3 \le j \le k}}
\cI_j
\end{equation}
and it is a principal ideal generated by 
\[
\displaystyle\prod_{\substack{j \in \IP \\ 
k-i+3 \le j\le k}}\nn_j\,.
\]
We know that $[\Delta,\Delta_{i-1}]$ is a subsheaf of $\Delta_i=[\Delta_{i-1},\Delta_{i-1}]$,
and thus to understand $[\Delta,\Delta_i]$ it suffices to know the Lie brackets
of the two generators $g_0$, $g_1$
of $\Delta$
with the single generator $g_i$ of the quotient sheaf $\Delta_i/\Delta_{i-1}$.
By definition $g_1=v_k$, and (\ref{vifjvanish}) tells us that $[g_1,g_i]=0$.
Thus by property~(\ref{gbrackets}) of Lemma~\ref{omnibus} we obtain 
(\ref{structureequation1}).
\end{proof}

\section{Focal order and vertical order} \label{fovo}

The notion of focal order which we lay out here is a special case of the notion
of order of a function with respect to a nonholonomic system, as explicated in
Section 2.1.1 of \cite{MR3308372} and in Section 4 of \cite{MR1421822}.
Our starting point and application are quite different, however;
we do not employ a metric, nor do we make use of privileged coordinates.
Thus we offer here a self-contained treatment, confined to the case
of the Goursat distribution $\Delta$ on a space in the monster tower.

Given a point $p\in S(k)$, a smooth function $a$ defined in some neighborhood, and a parameterized
smooth focal curve germ $C$ through $p$,
let $\#(a\cdot C)$ be the intersection number.
Concretely, if $\gamma\colon\Gamma\to C$
is the parametrization and
$t$ is the parameter (so that $\gamma(0)=p$), then
$\#(a\cdot C)$ is the order of vanishing of
$a \circ \gamma$ as function of $t$.
If the focal curve lies within the hypersurface $a=0$,
then the intersection number is $+\infty$.

We define \emph{the focal order} of $a$ to be
\[
o(a) = \min\{\#(a\cdot C)\},
\]
the minimum over all possible focal curve germs through $p$.
The focal order is a valuation, but here
the value $+\infty$ is impossible, 
since there are no integral hypersurfaces for
the focal distribution.
Note that in a family of curves, the intersection number
can only increase or stay the same under specialization;
in this sense the order is the generic value of the intersection number.

\begin{example}
Working at the point $p=(0,0;0,0,0,1) \in \cC(ooio)$ on $S(4)$,
and using the coordinates $x$, $y$, $y'=dy/dx$, $y''=dy'/dx$, $x'=dx/dy''$,
$x''=dx'/dy''$,
consider the function $a=y-\frac{1}{15}(y'')^{15}$.
We write
\begin{align*}
y''&=A_1t + A_2 t^2 + \cdots \\
x''&=1+B_1t+B_2t^2+\cdots
\end{align*}
where the ellipses indicate higher-order terms. 
By repeated integration we find that
\begin{align*}
x'&=\int x'' dy''
=A_1 t + \left(A_2+\tfrac{1}{2}A_1B_1\right)t^2+\cdots \\
x&=\int x' dy''
=\tfrac{1}{2}{A_1}^2\, t^2 + \left(A_1A_2+\tfrac{1}{6}{A_1}^2\,B_1\right)t^3 +\cdots \\
y'&=\int y'' dx
=\tfrac{1}{3}{A_1}^3\,t^3 + \left({A_1}^2\, A_2+\tfrac{1}{8}{A_1}^3\,B_1\right)t^4 +\cdots \\
y&=\int y' dx
=\tfrac{1}{15}{A_1}^5\, t^5 + \left(\tfrac{1}{3}{A_1}^4\, A_2+\tfrac{7}{144}{A_1}^5\, B_1\right)t^6 +\cdots \\
y-\tfrac{1}{15}(y'')^5 &=\tfrac{7}{144} {A_1}^5\, B_1 t^6 + \cdots
\end{align*}
so that the intersection number $\#(a\cdot C)$ is 6 whenever $A_1$ and $B_1$
don't vanish, and larger otherwise. Thus
$o(a)=6$.
(For those familiar with the notion of nonholonomic privileged coordinates,
this is the last coordinate in a system of such coordinates;
as already indicated, however, we make no use of this concept.)
\end{example}

The notion of order is easily extended to exact 1-forms:
we define $o(da)=o(a-a(p))$. 
Conversely, if we know $o(da)$ then we can compute $o(a)$:
if $a$ vanishes at $p$
we have $o(a)=o(da)$, and otherwise $o(a)=0$.
Using this extension, we can easily compute the focal orders
of the standard coordinates in a standard chart on a monster space.

\begin{example} \label{foexample}
We use the same chart $\cC(ooioii)$ as in Examples~\ref{focalvfexample},
\ref{standardbracketsexample}, and \ref{specialbasisexample}.
Working at the point 
$(x,y;y',y'',x',x'',y^{(3)},x^{(3)})=(0,0;0,0,0,1,0,0)$,
begin with the fact that the focal order of each active coordinate is 1,
and then compute as follows: 
\begin{align*}
o(dy^{(3)})&= 1   &o(y^{(3)})&= 1 \\
o(dx^{(3)})&= 1   &o(x^{(3)})&= 1 \\
o(dx'')&= o(x^{(3)}) + o(dy^{(3)}) = 2          &o(x'')&= 0 \\
o(dy'')&= o(y^{(3)}) + o(dx'') = 3          &o(y'')&= 3 \\
o(dx')&= o(x'') + o(dy'') = 3          &o(x')&= 3 \\
o(dx)&= o(x') + o(dy'') = 6          &o(x)&= 6 \\
o(dy')&= o(y'') + o(dx) = 9          &o(y')&= 9 \\
o(dy)&= o(y') + o(dx) = 15          &o(y)&= 15 
\end{align*}
The RVT code word of this point is $RRV\!RVV$.
\end{example}
\begin{example} \label{RRVTVVpoint}
If instead we work at the origin of this chart, whose code word is $RRVTVV$,
we find that
\begin{align*}
o(dy^{(3)})&=o(y^{(3)})=1    \\
o(dx^{(3)})&= o(x^{(3)})=1    \\
o(dx'')&= o(x'') = 2\\
o(dy'')&= o(y'') = 3 \\
o(dx')&= o(x')  = 5 \\
o(dx)&= o(x)  = 8 \\
o(dy')&= o(y')  =11 \\
o(dy)&= o(y) = 19 
\end{align*}
\end{example}

The general procedure is as follows.
\begin{enumerate}
\item The order of the differential of each active coordinate is 1.
\item Recursively, using the definitions of the coordinates in backwards order:
\begin{enumerate}
\item use the definition of the coordinate to infer the order of vanishing of the differential of a prior coordinate;
\item if the prior coordinate vanishes at $p$, then its order agrees with the order
of its differential, and otherwise is 0.
\end{enumerate}
\end{enumerate}

Given a parameterized focal curve germ $C$ through $p$, locally there is a focal vector
field $v_C$ whose integral curve through $p$ is $C$;
it is not unique, but its restriction to $C$ is unique.
For a smooth function germ $a$ in a neighborhood of $p$, we consider the Lie derivative $v_C(a)$.
\begin{lemma}
Suppose that $\#(a\cdot C)>0$. Then
\[
\#(v_C(a)\cdot C) = \#(a\cdot C)-1.
\]
\end{lemma}
\begin{proof}
This follows from the basic fact that 
for a curve germ parameterized by $\gamma\colon\Gamma\to C$,
we have
\[\tfrac{d}{dt}(a \circ \gamma)=v_C(a)\circ \gamma,
\]
and that differentiation reduces the order of vanishing by 1.
\end{proof}
\begin{corollary}
If $o(a)>0$, then $o(v_C(a))=o(a)-1$.
\end{corollary}

This leads to an alternative characterization of the focal order.
For each focal vector field $w$, we can differentiate $a$ repeatedly
with respect to this vector field until we obtain a function
with nonzero value at $p$; call this value the \emph{order of $a$ with respect to $w$}.
(If we never reach such a function, then we say the order is $+\infty$.)
Then $o(a)$ is the minimal order of $a$, taking the minimum over all focal vector fields.
One can even vary the choice of vector field at each step: $o(a)$ is also the least
number of focal derivatives required to take $a$ to a function with nonzero value.
This is the definition that one finds in \cite{MR1421822} and
\cite{MR3308372}.

For $2\le j \le k$, consider the divisor at infinity $I_j$.
Given a point 
$p\in S(k)$, we define its \emph{vertical order} $\VO_j$
to be the focal order of the function defining $I_j$,
i.e., the minimum intersection number of a regular
focal curve germ with this divisor;
this is independent of the choice of chart.
The  \emph{vertical orders vector} is
\[\
\VO(p) = (\VO_{2},\VO_{3},\dots,\VO_k).
\]
If $p\notin I_j$, then $\VO_j=0$;
in general we can compute it by computing the
focal order of the standard coordinate function defining $I_j$.

\begin{example}
Referring to Example~\ref{foexample}, we observe
that $(0,0;0,0,0,1,0,0)$ does not meet $I_2$ or $I_4$.
Furthermore the divisors $I_3$, $I_5$, and $I_6$ are defined by
the vanishing, respectively, of $x'$, $y^{(3)}$, and $x^{(3)}$.
Thus the vertical orders vector at $p$
is
\[
(0,3,0,1,1).
\]
\end{example}

To relate this notion to our earlier paper \cite{MR4887124},
observe that there we associated a vertical orders vector
to a curve germ $C$ on the base surface $S$ (and more
generally to a focal curve germ at some other level).
Starting with a curve $C$ that lifts to a regular curve germ
through $p\in S(k)$, we see that the two notions agree.
For a more extensive example, we refer to Example~31 of
\cite{MR4887124}, which shows that the vertical orders
vector associated to a point with RVT code word $RVVV\!RVT$
is $(6,3,3,0,2,0)$; since the word is not a Goursat word,
the first entry is nonzero.

\section{Sections of the small growth sheaves} \label{secondstructure}


Suppose that $k\ge3$.
Here we consider the small growth sheaves $\Delta^h$ on $S(k)$.
In Theorem~\ref{structuretheorem2} we will give a description of their sections;
it involves bounds on certain focal orders.
In the subsequent section, we will show that these bounds
are sharp, and this will allow us to give new characterizations
of the small growth invariants.

Choosing a point $p\in S(k)$ and using the entries in its vertical orders vector,
for $h\ge 2$ and
 $2 \le i\le \min\{h,k+1\}$ we define
\begin{equation} \label{ehidef}
e_{hi} =
\begin{cases}
 -(h-i)
+
\displaystyle\sum_{j=k-i+4}^k (i+j-k-3) \VO_j 
& \text{if this quantity is nonnegative,}
\\
\qquad 0 & \text{otherwise}.
\end{cases}
\end{equation}
The sum in (\ref{ehidef}) is zero if $i=2\text{ or } 3$;
hence $e_{h2}=e_{h3}=0$ in all cases.
Note that all the values $e_{hi}$ depend only on the Goursat code word of the specified point;
thus they are Goursat invariants.
It is helpful to arrange these invariants in a table, as in the following example.

\begin{example} \label{ehiRRVTVV}
For a point with Goursat word $RRVTVV$,
such as the point in Example~\ref{RRVTVVpoint},
here is a
table of values for $e_{hi}$.
To the right, we present the table of values
for a point whose Goursat word is $RRRVV$.
In Example~\ref{redrows} we explain why certain
rows are colored in red;
in Section~\ref{pathway}
we will explain
how we know the values of the small growth vector
listed in the rightmost column.
The juxtaposition of the two tables 
hints at a recursion for computing them;
we will explain this in Section~\ref{recursions},
where we introduce the notion of \emph{lifted Goursat word}.

\begin{center}
\begin{tabular}{|c|cccccc|c|}
\hline
\backslashbox{$h$}{$i$}&2&3&4&5&6&7&\begin{color}{blue}$\SG_h$\end{color}\\
\hline
\begin{color}{red}2\end{color} & \begin{color}{red}0\end{color} &  & & & & &\begin{color}{blue}3\end{color}\\
\begin{color}{red}3\end{color} & \begin{color}{red}0\end{color} & \begin{color}{red}0\end{color} & & & & &\begin{color}{blue}4\end{color}\\
4 &0& 0 & 1& & & &\begin{color}{blue}4\end{color}\\
\begin{color}{red}5\end{color} & \begin{color}{red}0\end{color} & \begin{color}{red}0\end{color} & \begin{color}{red}0\end{color}& \begin{color}{red}3\end{color}& & &\begin{color}{blue}5\end{color}\\
6 &0& 0 & 0& 2& 5& &\begin{color}{blue}5\end{color}\\
7 &0& 0 & 0&1 & 4& 12 &\begin{color}{blue}5\end{color}\\
\begin{color}{red}8\end{color} & \begin{color}{red}0\end{color} & \begin{color}{red}0\end{color} &\begin{color}{red}0\end{color}& \begin{color}{red}0\end{color}& \begin{color}{red}3\end{color}& \begin{color}{red}11\end{color} &\begin{color}{blue}6\end{color}\\
9 &0& 0 & 0& 0& 2&10 &\begin{color}{blue}6\end{color} \\
10 &0& 0 &0 &0 &1 &9 &\begin{color}{blue}6\end{color}\\
\begin{color}{red}11\end{color} & \begin{color}{red}0\end{color} & \begin{color}{red}0\end{color} & \begin{color}{red}0\end{color}& \begin{color}{red}0\end{color}& \begin{color}{red}0\end{color}&\begin{color}{red}8\end{color}&\begin{color}{blue}7\end{color}\\
12 &0& 0 & 0& 0& 0&7&\begin{color}{blue}7\end{color}\\
13 &0& 0 & 0& 0& 0&6&\begin{color}{blue}7\end{color}\\
14 &0& 0 & 0& 0& 0&5&\begin{color}{blue}7\end{color}\\
15 &0& 0 & 0& 0& 0&4&\begin{color}{blue}7\end{color}\\
16 &0& 0 & 0& 0& 0&3&\begin{color}{blue}7\end{color}\\
17 &0& 0 & 0& 0& 0&2&\begin{color}{blue}7\end{color}\\
18 &0& 0 & 0& 0& 0&1&\begin{color}{blue}7\end{color}\\
\begin{color}{red}19\end{color} & \begin{color}{red}0\end{color} & \begin{color}{red}0\end{color} & \begin{color}{red}0\end{color}& \begin{color}{red}0\end{color}& \begin{color}{red}0\end{color}&\begin{color}{red}0\end{color}&\begin{color}{blue}8\end{color}\\
\hline
\end{tabular}
\qquad \qquad \begin{tabular}{|c|ccccc|c|}
\hline
\backslashbox{$h$}{$i$}
&2&3&4&5&6&\begin{color}{blue}$\SG_h$\end{color}\\
\hline
\begin{color}{red}2\end{color} & \begin{color}{red}0\end{color} && & &  &\begin{color}{blue}3\end{color}\\
\begin{color}{red}3\end{color} & \begin{color}{red}0\end{color} &\begin{color}{red}0\end{color}& & &  &\begin{color}{blue}4\end{color}\\
4 &0& 0 & 1& & & \begin{color}{blue}4\end{color}\\
\begin{color}{red}5\end{color} &\begin{color}{red}0\end{color}& \begin{color}{red}0\end{color} & \begin{color}{red}0\end{color}& \begin{color}{red}3\end{color}&  &\begin{color}{blue}5\end{color}\\
6 &0& 0 & 0& 2& 5& \begin{color}{blue}5\end{color}\\
7 &0& 0 & 0&1 & 4& \begin{color}{blue}5\end{color}\\
\begin{color}{red}8\end{color} &\begin{color}{red}0\end{color}& \begin{color}{red}0\end{color} &\begin{color}{red}0\end{color}& \begin{color}{red}0\end{color}& \begin{color}{red}3\end{color}& \begin{color}{blue}6\end{color}\\
9 &0& 0 & 0& 0& 2 &\begin{color}{blue}6\end{color} \\
10 &0& 0 &0 &0 &1  &\begin{color}{blue}6\end{color}\\
\begin{color}{red}11\end{color} &\begin{color}{red}0\end{color}& \begin{color}{red}0\end{color} & \begin{color}{red}0\end{color}& \begin{color}{red}0\end{color}& \begin{color}{red}0\end{color} &\begin{color}{blue}7\end{color}\\
\hline
\end{tabular}
 \end{center}

\end{example}

Here are some
elementary properties:
\begin{equation} \label{obv}
\max\{e_{hi}-1,0 \} = e_{h+1,i}\,.
\end{equation}
\begin{equation} \label{alsoobv}
e_{hi} \ge e_{h+1,i}
\end{equation}
\begin{equation} \label{obvobv}
e_{h,i+1} \ge e_{hi}
\end{equation}
\begin{equation} \label{eprop}
e_{hi}+\sum_{j=k-i+3}^k \VO_j \ge e_{h+1,i+1}, \quad 
\text{with equality when $h=i$}\,.
\end{equation}

As remarked in Section~\ref{smallgrowthsection},
the subsheaves $\Delta_h$ and $\Delta^h$ 
of the tangent sheaf $\Theta_{S(k)}$
are equal for $h=1,2,3$.
To describe the other sheaves
in the small growth sequence,
we employ the ordered basis
$g_0, g_1, \dots, g_{k+1}$
developed in Lemma~\ref{omnibus}.

\begin{theorem} \label{structuretheorem2}
Consider a standard chart of $S(k)$ containing the point $p$.
For each $h\ge 3$,
each section of 
the sheaf $\Delta^h$ in this chart
may be written in the form
\begin{equation} \label{structure2}
w + \sum_{i=4}^{\min\{h,k+1\}} c_{hi} g_i,
\end{equation}
where $w$ is a section of $\Delta_3$, and
each $c_{hi}$ is a function whose focal order $o(c_{hi})$ at $p$ is at least $e_{hi}$.
\end{theorem}

\begin{proof}
We prove the theorem by induction on $h$.
For the base case $h=3$, the sum in (\ref{structure2}) is empty
and the assertion is that $\Delta^3=\Delta_3$, which we already know.
The theorem thus being clear for $k < 3$, we henceforth assume that $k\ge 3$.

Consider a section of $\Delta^h$; by the inductive hypothesis it may be written
as in (\ref{structure2}). Bracketing this section with a focal vector field $z$ (i.e., a section of $\Delta$) yields
\begin{equation} \label{3termtypes}
[z,w] + \sum_{i=4}^{\min\{h,k+1\}} \left(c_{hi}[z, g_i] + z(c_{hi})g_i \right).
\end{equation}
Note we have used the product rule (\ref{productrule}).
We will separately analyze the three
types of terms in (\ref{3termtypes}),
showing that each contribution satisfies the required inequality
on the orders of the coefficients.
To begin,
by Corollary~\ref{oldstructuretheorem1}
the vector field $[z,w]$ is a section of $\Delta_3 +\cI_k \Delta_4$,
and each section of $\cI_k$ on our chart is a function with
focal order at least $\VO_k \ge e_{h,4}$.

Assume that $h < k+1$.
We now deal with the second
sort of term in (\ref{3termtypes}).
By Corollary~\ref{oldstructuretheorem1},
each $[z,g_i]$ is a section of
\[\Delta_i + \left(\displaystyle\prod_{j=k-i+3}^k \cI_j\right)\Delta_{i+1}\,,\] 
and thus may be written as 
\[
w_i + \left(\displaystyle\prod_{j=k-i+3}^k c_j\right)g_{i+1}\,,
 \]
 where $w_i$ is a section of $\Delta_i$
 and 
 each function $c_j$ is a section of $\cI_j$.
 Writing
 \[
 w_i = w_3 + \sum_{j=4}^{i} a_{j} g_j
 \]
with $w_3$ a section of $\Delta_3$,
we remark that
\[
o(c_{hi}a_j) 
\ge
o(c_{hi})
\ge
e_{hi}
\ge
e_{h+1,j}
\]
for all $j$; here
we have used the inductive hypothesis
together with
properties
(\ref{alsoobv}) and (\ref{obvobv}).
 Observe that by (\ref{eprop}) we know that
 \[
 o\left(\displaystyle c_{hi}\prod_{j=k-i+3}^k c_j \right) \ge e_{hi} + \displaystyle \sum_{j=k-i+3}^k \VO_j \ge e_{h+1,i+1}\,.
 \]
 
For the third type of term in (\ref{3termtypes})
we observe that
\[o(z(c_{hi})) \ge \max\{o(c_{hi})-1,0\} \ge \max\{e_{hi}-1,0 \} = e_{h+1,i}
\]
using the inductive hypothesis
and property (\ref{obv}).

If $h \ge k+1$, the same arguments apply, except that $[z, g_{k+1}]$ is already a section of the full tangent bundle $\Delta_{k+1}$.
\end{proof}

\begin{example} \label{st2example}
As in Example~\ref{ehiRRVTVV}, we consider a point $p\in S(6)$
in chart $\cC(ooioii)$ whose code word is $RRVTVV$.
Applying Theorem~\ref{structuretheorem2} with $h=7$, we see that
in this chart each section of $\Delta^7$
may be written as
\[
w + c_{74} g_4 + c_{75} g_5  + c_{76} g_6  + c_{77} g_7, 
\]
where $w$ is a section of $\Delta_3$,
and
with $o(c_{75})\ge 1$, $o(c_{76})\ge 4$, and $o(c_{77})\ge 12$.
\end{example}

For each $i$ with $2 \le i\le k+1$, let $b_i$
be the smallest integer for which $e_{b_i,i}=0$.
Using the definition in (\ref{ehidef}), we see that
\begin{equation} \label{biformula}
b_i = i + \sum_{j=k-i+4}^k (i+j-k-3) \VO_j .
\end{equation}
We call $(b_2,b_3,\dots,b_{k+1})$ the \emph{$b$ vector}.

\begin{example} \label{redrows}
Consider the table on the left in
Example~\ref{ehiRRVTVV}. To directly use the definition of $b_i$
we look in each column for the first appearance of a zero,
concluding that $b_2=2$, $b_3=3$, $b_4=5$, $b_5=8$, $b_6=11$, and $b_7=19$.
Alternatively, using formula (\ref{biformula}),
we have, e.g.,
\[
b_7 = 7 + \VO_3 + 2\VO_4 + 3\VO_5 + 4\VO_6 = 7 + 5 + 0 + 3 + 4 = 19.
\]
In both tables, the rows indexed by $b_i$'s
are highlighted in red.
\end{example}

\begin{corollary} \label{comparison}
\phantom{x}
\begin{enumerate}
\item
For each $h\ge 2$,
\[\SG_h \le 2 + \text{the number of zero entries in row $h$ of the
table of $e_{hi}$ values}.\] 
\item
For $2\le i \le k+1$ we have
$\beta_{i+1} \ge b_i$.
\end{enumerate}
\end{corollary}

\begin{proof}
\phantom{x}
\begin{enumerate}
\item
Consider a section of $\Delta^h$, 
as in display (\ref{structure2}).
When we evaluate this section at $p$,
any coefficient $c_{hi}$ for which $e_{hi}>0$ will vanish.
\item \label{betavsb}
This is now immediate from the definitions of $\beta_i$ and $b_i$.
\end{enumerate}
\end{proof}

\section{Calculation pathways} \label{pathway}

In this section, we will show that the inequalities in 
Corollary~\ref{comparison} are actually equalities.
To do this, we produce some specific sections of the sheaves $\Delta^h$,
by following certain calculation pathways.

\begin{theorem}
Assume the circumstances of Theorem~\ref{structuretheorem2}.
For each $h\ge 3$
and for each $i$ with $3 \le i\le \min\{h,k+1\}$,
there is a section $f_{hi}$ of $\Delta^h$ 
for which, when we write it as in (\ref{structure2}),
the coefficient $c_{hi}$
has focal order exactly $e_{hi}$.
\end{theorem}

\begin{proof}
We first remark that it suffices to produce $f_{hi}$
when the quantity on the top line of (\ref{ehidef}) 
is nonnegative, since in the remaining cases we
may set  $f_{hi}=f_{h-1,i}$.
We will obtain each of the other required sections by beginning
with $f_{33}=g_3$ and repeatedly Lie-bracketing on the left by
either $g_0$ or $g_1$, with the exact sequence of calculations
to be specified below. 
The basis $g_0, g_1, \dots, g_{k+1}$ has been chosen so that 
the results of these calculations have only positive coefficients,
and thus it suffices to exhibit a single term of $f_{hi}$ 
of the form $ag_i$ in which $a$ has the required focal order $e_{hi}$,
since there is no possibility of cancellation.

To produce $f_{ii}$, we use only $g_0$:
\begin{align*}
f_{33} & = g_3 \\
f_{44} & = [g_0,f_{33}] \\
f_{55} & = [g_0,f_{44}] \\
&\text{etc.}
\end{align*}

We claim that $f_{ii}$ includes the term
\begin{equation} \label{exhibit}
t_i:=\left(\prod_{\substack{j=k-i+4 \\ j \in \IP}}^k  \nn_j^{i+j-k-3} \right) g_i\,
\end{equation}
which visibly has the required focal order
\[
e_{ii}=\sum_{j=k-i+4}^k (i+j-k-3) \VO_j .
\]
This is clear by induction, beginning with
$f_{33}=g_3$
(with the product being interpreted by the usual convention to be $1$).
In the inductive step
we see that if we apply the product rule (\ref{productrule}) to $[g_0,t_i]$
we obtain two terms,
one of which is
\begin{align*}
\left(\prod_{\substack{j=k-i+4 \\ j \in \IP}}^k  \nn_j^{i+j-k-3} \right) 
[g_0,g_i]
&= \left(\prod_{\substack{j=k-i+4 \\ j \in \IP}}^k  \nn_j^{i+j-k-3} \right) 
\left(\displaystyle
\prod_{\substack{j=k-i+3 \\ j \in \IP}}^k
\nn_j
\right)
g_{i+1}
\\
&=
\left(\prod_{\substack{j=k-(i+1)+4 \\ j \in \IP}}^k  \nn_j^{(i+1)+j-k-3} \right)
g_{i+1}\,.
\end{align*}
(Note that we have used property (\ref{gbrackets}) of Lemma~\ref{omnibus} in this computation.)

In general, we define the vector field $f_{hi}$
by fixing $i$ and 
inducting on $h$, beginning with $f_{ii}$.
Having singled out a term 
\begin{equation}\label{aterm}
t_h=ag_i
\end{equation}
in $f_{hi}$
with $o(a)=e_{hi}$,
we examine the coefficient $a$
to see whether it includes the coordinate function $\nn_k$.
If so, we define $f_{h+1,i}=[g_1,f_{hi}]$;
then by the product rule $f_{h+1,i}$ includes the term
\[
t_{h+1}:=g_1(a)g_i = v_k(a)g_i = \frac{\partial}{\partial \nn_k}(a) g_i \,,
\]
in which the coefficient has focal order $o(a)-1=e_{hi}-1=e_{h+1,i}$.
If the coefficient $a$ in (\ref{aterm}) does not include $\nn_k$,
then we define $f_{h+1,i}=[g_0,f_{hi}]$.
By the product rule $f_{h+1,i}$ includes $g_0(a)g_i = f_k(a)g_i$,
which may consist of multiple terms. There is, however, at least one term,
and each one is of the form $a'g_i$, where the coefficient has focal order
$o(a')=o(a)-1=e_{hi}-1=e_{h+1,i}$.
Select one of these terms and call it $t_{h+1}$.
This induction continues until we reach the first zero value
for $e_{hi}$, i.e., until we reach $h=b_i$.
\end{proof}

\begin{corollary} \label{bridge}
The inequalities of Corollary~\ref{comparison} are in fact equalities:
\begin{enumerate}
\item
For each $h\ge 2$,
\[\SG_h = 2 + \text{the number of zero entries in row $h$ of the
table of $e_{hi}$ values}.\] 
\item
For $2\le i \le k+1$ we have
$\beta_{i+1} = b_i$.
\end{enumerate}
\end{corollary}

\begin{example}
As in Example~\ref{ehiRRVTVV}, 
we work in chart $\cC(ooioii)$.
For the $RRVTVV$ point at the origin, here is a calculation 
leading to the vector field $f_{19,7}$.

\begin{center}
\begin{tabular}{c|c|c}
vector field & term in this vector field & focal order of its coefficient \\
\hline
$f_{33}=g_3$ 		& $t_3=g_3$ & 0 \\
$f_{44}=[g_0,f_{33}]$ 	& $t_4=\nn_6 g_4$ & 1 \\
$f_{55}=[g_0,f_{44}]$ 	& $t_5=\nn_5 {\nn_6}^2 g_5$ & 3 \\
$f_{66}=[g_0,f_{55}]$	& $t_6={\nn_5}^2 {\nn_6}^3 g_6$ & 5 \\
$f_{77}=[g_0,f_{66}]$	& $t_7=\nn_3{\nn_5}^3 {\nn_6}^4 g_7$ & 12 \\
\hline
$f_{87}=[g_1,f_{77}]$ & $t_8=4\,\nn_3 {\nn_5}^3 {\nn_6}^3 g_7$ & 11 \\
$f_{97}=[g_1,f_{87}]$ & $t_9=12\,\nn_3 {\nn_5}^3 {\nn_6}^2 g_7$ & 10 \\
$f_{10,7}=[g_1,f_{97}]$ & $t_{10}=24\,\nn_3 {\nn_5}^3 \nn_6 g_7$ & 9 \\
$f_{11,7}=[g_1,f_{10,7}]$ & $t_{11}=24\,\nn_3 {\nn_5}^3 g_7$ & 8 \\
$f_{12,7}=[g_0,f_{11,7}]$ & $t_{12}=24\,\nn_4 {\nn_5}^4 \nn_6 g_7$ & 7 \\
$f_{13,7}=[g_1,f_{12,7}]$ & $t_{13}=24\,\nn_4 {\nn_5}^4 g_7$ & 6 \\
$f_{14,7}=[g_0,f_{13,7}]$ & $t_{14}=24\,{\nn_5}^4 \nn_6 g_7$ & 5 \\
$f_{15,7}=[g_1,f_{14,7}]$ & $t_{15}=24\,{\nn_5}^4 g_7$ & 4 \\
$f_{16,7}=[g_0,f_{15,7}]$ & $t_{16}=96\,{\nn_5}^3 g_7$ & 3 \\
$f_{17,7}=[g_0,f_{16,7}]$ & $t_{17}=288\,{\nn_5}^2 g_7$ & 2 \\
$f_{18,7}=[g_0,f_{17,7}]$ & $t_{18}=576\,\nn_5 g_7$ & 1 \\
$f_{19,7}=[g_0,f_{18,7}]$ & $t_{19}=576\, g_7$ & 0 \\
\end{tabular}
\end{center}

\bigskip
\noindent
(By a laborious calculation one discovers that in fact $f_{19,7}=t_{19}$.)

\end{example}

\section{Relating the structural and small growth invariants} \label{tsgi}

Corollary~\ref{bridge} provides a bridge between the
structural invariants and the small growth invariants,
which we now explain further.
Looking at a point $p\in S(k)$,
we consider three structural invariants:
\begin{enumerate}
\item
The $b$ vector 
\[(b_2,b_3,\dots,b_{k+1})\]
has been defined in Section~\ref{secondstructure}.
\item
From the multiplicity sequence $m_0,m_1,\dots$ 
of any curve on $S$ whose lift is a
regular focal curve through $p$,
we extract the entries $m_1$ through $m_{k-1}$,
writing them in reverse order.
We call 
\[
(m_{k-1},m_{k-2},\dots,m_1)
\]
the \emph{multiplicity vector}.
For the definition of the multiplicity sequence,
see Section~13 of \cite{MR4887124} or 
Section 3.5 of \cite{MR2107253}.
The omitted value $m_0$ is not a Goursat invariant,
and at the other end we have $m_i=1$ for $i\ge k$.
\item
We consider
the restricted vertical orders vector 
\[
(\VO_k, \VO_{k-1},\dots,\VO_3),
\]
writing the entries in reverse order from the convention used in \cite{MR4887124}.
\end{enumerate}
There are three corresponding small growth invariants,
namely the beta vector and its first two derived vectors,
denoted $\der$ and $\der^2$.

By Theorem~33 of \cite{MR4887124}, we have
$m_i-m_{i+1}=\VO_{i+2}$;
this tells us that the (reversed) restricted vertical orders vector
is the derived vector of the multiplicity vector, i.e.,
its vector of successive differences.\footnote{We take the opportunity to correct a slight error in the range of applicability
of Theorem~33 in that earlier paper; the appropriate range there
should be $k\le j \le r-2$.}
Going in the opposite direction,
we see that the vertical orders are obtained from
the multiplicities by accumulation:
\begin{equation} \label{accum}
m_i=1+\sum_{j=i+2}^k \VO_j
\end{equation}
(since we automatically have $m_{k-1}=1$). 
Using the entries in the $b$ vector
as specified by (\ref{biformula}), 
we compute that
\[
b_{i+1}-b_i=1+\sum_{j=k-i+3}^k \VO_j 
\]
and then compare with (\ref{accum})
to conclude that $b_{i+1}-b_i=m_{k-i+1}$.
This tells us that the multiplicity vector
is the derived vector of the $b$ vector.
From these remarks and Corollary~\ref{bridge}
we obtain the following result.

\begin{theorem}
The $b$ vector is
obtained from the beta vector by removing its initial entry.
Therefore the multiplicity vector is obtained
from $\der$, the derived vector of the beta vector, 
by removing its initial entry,
and the (reversed) restricted vertical orders vector
is obtained from $\der^2$ by removing its initial entry.
Conversely one obtains $\beta$, $\der$, and $\der^2$
from the $b$ vector, the multiplicity vector, and the
reversed restricted vertical orders vector by
inserting the initial entries 1, 1, 0 (respectively).
\end{theorem}

\begin{example}
For a point with RVT code word $RRVTVV$, we have $k=6$ and
\begin{align*}
(\beta_2,\dots, \beta_8)&= (1,2,3,5,8,11,19)   &(b_2,\dots, b_7)&= (2,3,5,8,11,19) \\
(\der_3,\dots, \der_8)&= (1,1,2,3,3,8)            &(m_5,\dots, m_1)&= (1,2,3,3,8) \\
({\der^2}_4,\dots, {\der^2}_8)&= (0,1,1,0,5)     &(\VO_6,\dots, \VO_3)&= (1,1,0,5).
\end{align*}
\end{example}

\section{Recursions} \label{recursions}

Here we address the issue of effective calculations.
If we know any one of the six invariants of Section~\ref{tsgi},
we can easily obtain the other five; thus the issue
is how to gain the first toehold.

In Section~\ref{secondstructure} we presented
a method for determining the vertical orders:
working in a chart containing our selected point $p\in S(k)$, calculate
the focal orders of the standard coordinates;
each vertical order equals one of these focal orders or is zero.
This is a rather cumbersome method, however, involving
parallel calculations for the focal orders of
both the coordinates and their differentials.

The proximity diagram of Section~13 of \cite{MR4887124}
allows one to easily compute the multiplicity sequence. This seems to be the most efficient method for obtaining all six of our invariants.
(Regarding proximity diagrams, see also Section 3.5 of \cite{MR2107253}, but note that Wall does
not use code words.) 

Implicit in the use of the proximity diagram is a front-end
recursion, meaning a recursion
for computing an invariant that uses alterations
to the front end of the code word.
Given a Goursat word $W$, we obtain its \emph{lifted Goursat word} $\LG(W)$
by the following procedure:
\begin{itemize}
\item
Remove the first symbol $R$.
\item
If the new second symbol is $V$ (i.e., if the truncated word fails to be Goursat), replace it by $R$,
and likewise replace any immediately succeeding $T$'s
by $R$'s. Stop when you reach the next $R$, the next $V$, or the end of the word.
\end{itemize}

\noindent
The terminology comes from a consideration of focal curves
on the monster spaces and their lifts into the monster tower.
We introduced a similar procedure in
Section~10.5 of \cite{MR4887124}, where we defined the lifted word associated to an arbitrary RVT code word; here, by contrast, we begin and end with Goursat words.

The proximity diagram for $W$ is obtained from the
proximity diagram for $\LG(W)$ as follows:
\begin{itemize}
    \item 
    Put a new vertex to the left (in position 0)
    and draw a horizontal edge to the vertex in position 1.
    \item
    From the vertex in position 1, draw edges rightward
    to the vertices whose labels have changed (if any);
    these vertices will be labeled by a block $VT^\tau$.
\end{itemize}
This procedure is shown in Figure~\ref{prox},
which also illustrates the recursive rule
for calculating the multiplicity $m_i$:
sum the multiplicities on the vertices proximate to vertex $i$.
(The leftmost vertex is vertex 0.)
Appending $m_1$ to the end of the derived vector for $\LG(W)$ gives us the derived  vector for $W$.
Similarly, one obtains each of our five other invariants for a Goursat word $W$ by appending the appropriate value to the end of the same invariant for $\LG(W)$;
the easiest way to find this value seems to be to first work out the multiplicities.

\begin{figure}[!htb] 
\begin{center}
\includegraphics[scale=0.8]{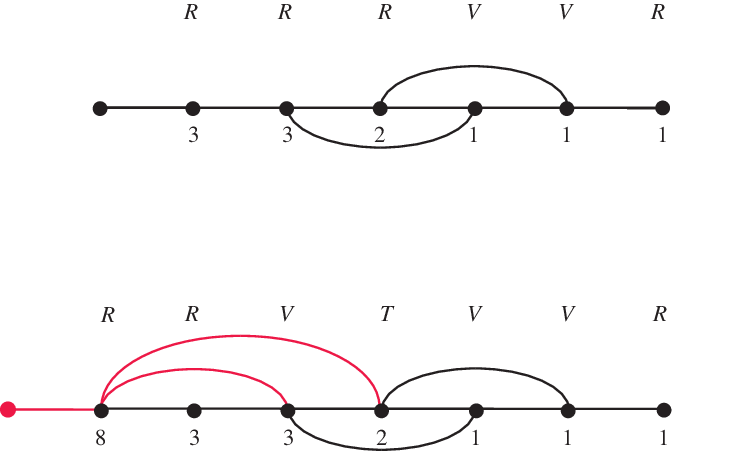}
\end{center}
\caption{Obtaining the proximity diagram
for the Goursat word $RRVTVVR$ from the proximity diagram
for its lifted Goursat word $RRRVVR$;
the new multiplicity is $m_1=3+3+2$.}
\label{prox}
\end{figure}

\begin{example}
Here is the front-end calculation of the derived vector of
$RRVTRRRVTTTV$:
\begin{align*}
\der(RR) &= (1,1) \\
\der(RRV) &= (1,1,2) \\
\der(RRRRRV) &= (1,1,2,2,2,2) \\
\der(RRVTTTV) &= (1,1,2,2,2,2,9) \\
\der(RRRRRRVTTTV) &= (1,1,2,2,2,2,9,9,9,9,9) \\
\der(RRVTRRRVTTTV) &= (1,1,2,2,2,2,9,9,9,9,9,27) \\
\end{align*}
\end{example}

There are also back-end recursions for our invariants,
dating back to the work of Jean \cite{MR1411581},
who stated such a recursion for the beta vector.
He stated and proved it
while studying the configuration space for a physical
system: a truck pulling multiple trailers.
This was before the later work of 
Montgomery and Zhitomirskii
\cite{MR1841129}, who showed that
the truck-and-trailers configuration space
is a universal space for Goursat distributions.
More precisely, the configuration space for a truck with $n$ trailers is the oriented version
of the monster space $\mathbb{R}^2(n+1)$ over the real plane;
it is thus an unramified cover of degree $2^{n+1}$ over
 $\mathbb{R}^2(n+1)$.
Jean's recursion is stated
in terms of special angles between trailers,
and his proof of the recursion is achieved via an
elaborate recursive calculation in trigonometry.
Jean's work also predates the introduction of code words
for coarse classification of Goursat distributions;
see \cite{MR2172057} for an early instance.

Here we will show that Jean's back-end recursion
is a consequence of our front-end recursion,
thus extending its validity to all Goursat
distributions in all three settings of Section~\ref{3settings}.

\begin{theorem}
For every nonempty Goursat word $\beta_2(W)=1$,
and $\beta_3(W)=2$ for every word of length at least 2.
Letting $X$ represent any single symbol, we have
these recursive formulas:
\begin{enumerate}
\item 
$\beta_j(W\!R)=1+\beta_{j-1}(W)$,
\item 
$\beta_j(W\!XV)=\beta_{j-1}(W\!X)+\beta_{j-2}(W)$,
\item 
$\beta_j(W\!XT)=2\beta_{j-1}(W\!X)-\beta_{j-2}(W)$.
\end{enumerate}
\end{theorem}
\begin{proof}
We first observe that Jean's recursions are equivalent
to the following recursions for computing the derived vector, which we proceed to prove:
\begin{enumerate}
\item \label{addR}
$\der_j(W\!R)= \der_{j-1}(W)$,
\item \label{addV}
$\der_j(W\!XV)=\der_{j-1}(W\!X)+\der_{j-2}(W)$,
\item \label{addT}
$\der_j(W\!XT)=2\der_{j-1}(W\!X)-\der_{j-2}(W)$.
\end{enumerate}

Using the derived vector to label the vertices of
the proximity diagram, let $\der_v(W)$ indicate the
component of $\der_v(W)$ labeling vertex $v$.
Note that $\der$ is essentially the multiplicity vector;
thus we use the same rules to compute it from the 
diagram.
One obtains the proximity diagram of
$W\!R$ from that of $W$ by adding a single vertex
at the right, together with a single horizontal edge.
The label at this vertex is 1, and the label at
each other vertex is unchanged. This establishes
(\ref{addR}).

To establish (\ref{addV}), we consider three cases.
\begin{itemize}
\item 
Compare the right ends of the proximity diagrams
of $W\!RV$, $W\!R$, and $W$:
\begin{center}
\includegraphics[scale=0.85]{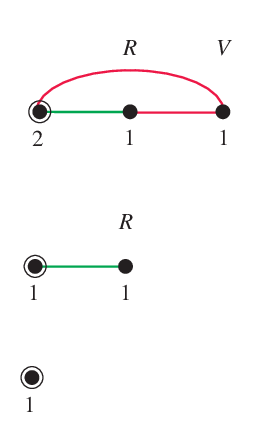}
\end{center}
\noindent Observe that at the circled vertex we have 
$\der_v(W\!RV)=\der_v(W\!R)+\der_v(W)$.
All the undrawn edges are common to all three diagrams;
thus the same equation holds for all vertices further
to the left.
\item
Compare the right ends of the proximity diagrams
of $WVV$, $WV$, and $W$:
\begin{center}
\includegraphics[scale=0.85]{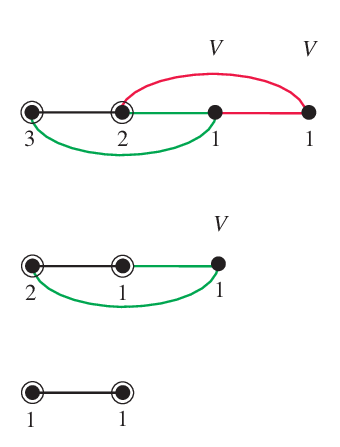}
\end{center}
\noindent Note that we show only those edges that begin and end
within the indicated portion of the diagrams;
there may or may not be an additional edge coming in from
the left, but it is common to all three diagrams.
Observe the equality 
$\der_v(WVV)=\der_v(WV)+\der_v(W)$
at the circled vertices.
This must persist for all vertices further to the left of these. 
\item
Compare the right ends of the proximity diagrams
of $WTV$, $WT$, and $W$:
\begin{center}
\includegraphics[scale=0.8]{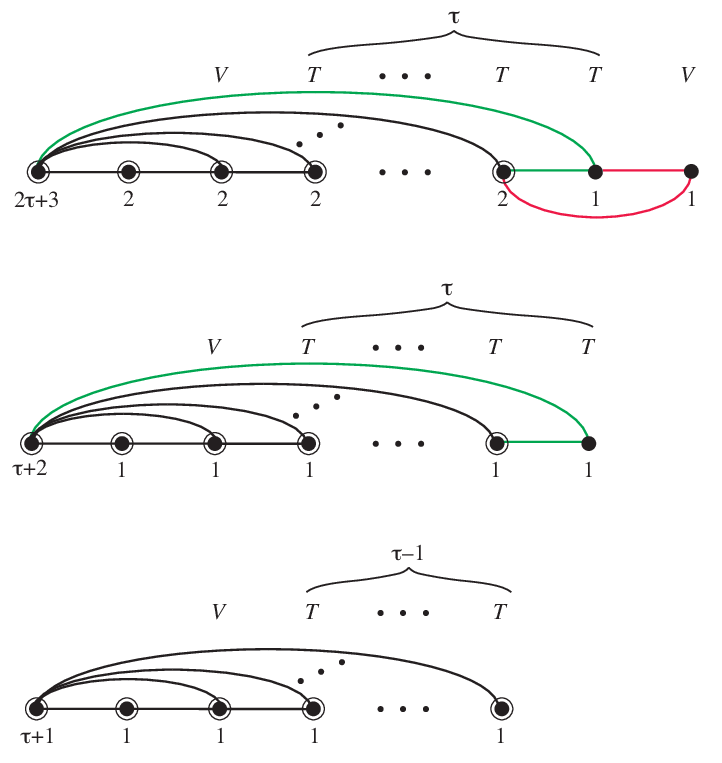}
\end{center}
\noindent Again we show only those edges that begin and end within the indicated portion of the diagrams. Observe that $W$ ends with $VT^{\tau-1}$ for some positive integer $\tau$. At the circled vertices we have $\der_v(WTV) = \der_v(WT) +\der_v(W)$; thus the same equation must hold at every vertex further to the left.

\end{itemize}

To establish (\ref{addT}), we consider two cases.

\begin{itemize}
\item
Compare the right ends of the proximity diagrams
of $WVT$, $WV$, and $W$:
\begin{center}
\includegraphics[scale=0.8]{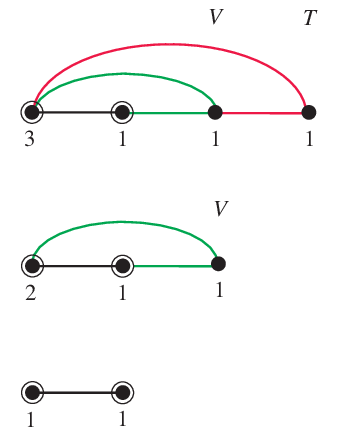}
\end{center}
The equality $\der_v(WVT) = 2\der_v(WV) - \der_v(W)$ holds at the two circled vertices, and thus at every vertex further to the left.
\item
Compare the right ends of the proximity diagrams of $WTT$, $WT$, and $W$: 
\begin{center}
\includegraphics[scale=0.8]{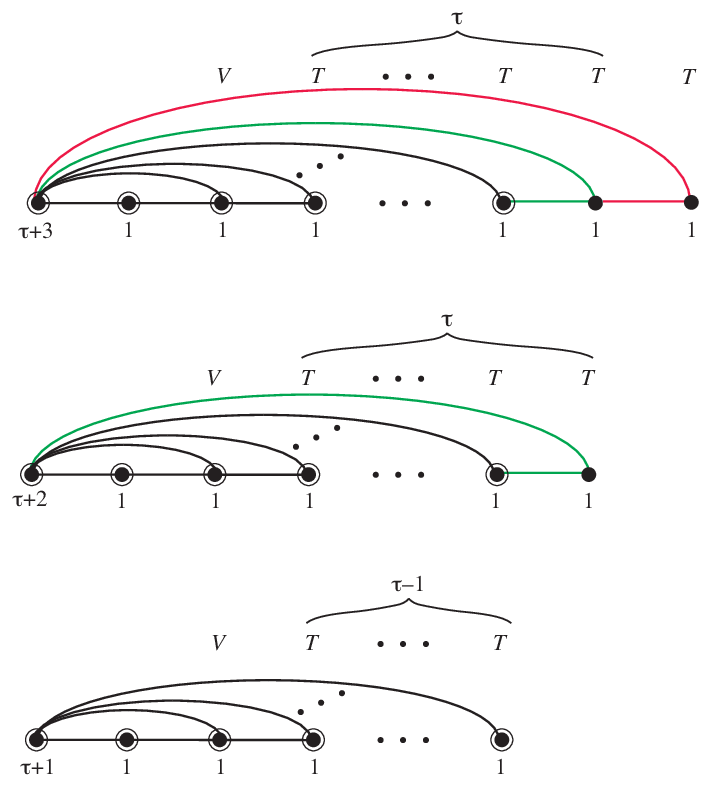}
\end{center}
The word $W$ ends with $VT^{\tau-1}$ for some positive integer $\tau$. At the circled vertices we have 
$\der_v(WTT) = 2\der_v(WT) - \der_v(W)$; thus the same equation must hold at every vertex further to the left.
\end{itemize}
\end{proof}

\begin{example}
Here is the back-end calculation of the derived vector of
$RRVTRRRVTTTV$:
\begin{align*}
\der(RR)&=(1,1)\\
\der(RRV)&=(1,1,2)\\
\der(RRVT)&=(1,1,1,3)\\
\der(RRVTR)&=(1,1,1,1,3)\\
\der(RRVTRR)&=(1,1,1,1,1,3)\\
\der(RRVTRRR)&=(1,1,1,1,1,1,3)\\
\der(RRVTRRRV)&=(1,1,2,2,2,2,2,6)\\
\der(RRVTRRRVT)&=(1,1,1,3,3,3,3,3,9)\\
\der(RRVTRRRVTT)&=(1,1,1,1,4,4,4,4,4,12)\\
\der(RRVTRRRVTTT)&=(1,1,1,1,1,5,5,5,5,5,15)\\
\der(RRVTRRRVTTTV)&=(1,1,2,2,2,2,9,9,9,9,9,27)\\
\end{align*}

\end{example}

For completeness, here is the back-end recursion for the second derived vector
(i.e., for the vertical orders).
To begin, we have
\[
({\der^2}_4(W),{\der^2}_5(W)) = 
\begin{cases}
(0,1) & \text{if $W$ ends with $V$}\\
(0,0) & \text{otherwise.}
		 \end{cases}
\]
For the other entries use these formulas, identical to those for the derived vector:
\begin{enumerate}
\item 
${\der^2}_j(W\!R)= {\der^2}_{j-1}(W)$,
\item 
${\der^2}_j(W\!XV)={\der^2}_{j-1}(W\!X)+{\der^2}_{j-2}(W)$,

\item 
${\der^2}_j(W\!XT)=2\,{\der^2}_{j-1}(W\!X)-{\der^2}_{j-2}(W)$.
\end{enumerate}

\section{Degree of nonholonomy via the Puiseux characteristic} \label{NHPC}

We end with an observation concerning the degree of nonholonomy.
For a point $p\in S(k)$, 
let $W$ be the RVT code word at $p$.
As we know, $W$ determines the beta vector
of the focal focal distribution $\Delta(k)$ at $p$.
We denote it by
$\beta(W)=(\beta_2,\dots, \beta_{k+2})$;
its last entry $\beta_{k+2}$ is the degree of nonholonomy.
Let $\PC(W)=[\lambda_0;\lambda_1, \dots, \lambda_g]$ be the Puiseux characteristic, as defined in Section~12.5 of 
\cite{MR4887124}. Recall that this is the Puiseux characteristic
associated to any regular focal curve germ passing through $p$,
which in turn is the Puiseux characteristic associated to the
(usually singular) curve germ on $S$ obtained by projecting.
Its initial entry $\lambda_0$ is the multiplicity $m_0$ of this curve.
Note that $m_0$ is not a Goursat invariant,
as the following example illustrates.

\begin{example}
The germ of the focal distribution $\Delta(5)$ at a point with code word
$RVTRV$ is equivalent to the focal distribution at some other point
with Goursat code word $RRRRV$.
The associated Puiseux characteristics are $[6;8,9]$ and $[2;9]$.
\end{example}

\begin{theorem}\label{last}\phantom{x}
\begin{enumerate}
\item \label{lastlast}
If $W$ ends with a critical symbol $V$ or $T$,
then the last entry in the Puiseux characteristic is the degree of
nonholonomy: $\lambda_g=\beta_{k+2}$. 
\item \label{lastlastlast}
More generally, if $W$ ends with a critical symbol followed
by a string of $r$ occurrences of $R$, 
then $\lambda_g+r=\beta_{k+2}$.
\end{enumerate}
\end{theorem}
\begin{proof}
To prove (\ref{lastlast}), we first observe that
if $W$ is $RVT^{\tau}$, then the Puiseux characteristic is
$[\tau+2;\tau+3]$ and the beta vector is $(1,2,\dots,\tau+3)$.
These words provide the base cases for a recursion.
If $W$ is not of this form, then its lifted word $L(W)$
again ends in a critical symbol.
Since the beta vector is the accumulation vector for the multiplicities,
we know
\begin{equation} \label{betambeta}
\beta_{k+2}(W)=m_0(L(W))+\beta_{k+1}(L(W)).
\end{equation}
There is a similar result for the last entry of the Puiseux
characteristic, according to Theorem~24 of \cite{MR4887124};
in all three cases of that theorem we see that
\[
\lambda_g(W)=\lambda_0(L(W))+\lambda_{\text{last}}(L(W)).
\]
(Here $\lambda_{\text{last}}$ is either $\lambda_g$ or $\lambda_{g-1}$.)
Thus if the desired equation holds for $L(W)$ it likewise holds for $W$.

Part (\ref{lastlastlast}) follows from the observation that
adding an $R$ to the end of a word doesn't alter the Puiseux
characteristic, while increasing the degree of holonomy by 1.
\end{proof}

The proof is elementary, but it may raise a question. As we have said, $m_0$ is not a Goursat invariant, but it appears
in equation (\ref{betambeta}) as the difference of two
Goursat invariants; how can this be?
The apparent contradiction is resolved by observing
that all the RVT code words associated to a particular Goursat
distribution have the same lifted word; thus although
the distribution does not determine $m_0(W)$
it does determine $m_0(L(W))$. To say this another way,
the quantity $m_0(L(W))$ is the same as $m_1(W)$,
and we know that $m_1$ is a Goursat invariant.

Theorem~\ref{last} is implicit in \cite{MR3152113}.
Our discussion here concerns just the initial entry and last entry
of the Puiseux characteristic; the cited paper also
treats the more intricate relations between
the other entries of the Puiseux characteristic
and the multiplicity sequence.


\bibliography{mybib}{}
\bibliographystyle{plain}


\end{document}